\documentclass[10pt,a4paper]{article}

\usepackage{a4wide}
\usepackage{amsthm}
\usepackage{amsfonts}
\usepackage{amsmath}
\usepackage[english]{babel}

\usepackage{faktor}
\usepackage{etex}
\usepackage{hyperref}
\usepackage{multimedia}
\usepackage[all]{xy}
\usepackage{amssymb}
\usepackage{graphicx,textpos}
\usepackage[dvipsnames]{xcolor}
\usepackage{helvet}
\usepackage{stmaryrd}
\usepackage{enumerate}
\usepackage{todonotes}
\usepackage{pdfsync}
\usepackage{dsfont}
\usepackage[shortcuts]{extdash}
\usepackage[all]{xy}
  \newdir{ >}{{}*!/-9pt/@{>}}

\newcommand{\R}{\mathbb{R}}

\newcommand{\C}{\mathbb{C}}

\newcommand{\I}{\mathds{1}}

\newcommand{\rst}[1]{\ensuremath{{\mathbin\mid}\raise-.5ex\hbox{$#1$}}}
\newcommand{\lie}{\mathfrak{g}}
\newcommand{\lien}{\mathfrak{n}}
\newcommand{\lieh}{\mathfrak{h}}

\author{Jonas Der\'e\thanks{The author was supported by a postdoctoral fellowship of the Research Foundation -- Flanders (FWO).}}
\title{\textbf{Orthogonal bi-invariant complex structures \\ on metric Lie algebras}}
\date{\vspace{-1cm}}

\newtheorem{Def}{Definition}[section]
\newtheorem{Ex}[Def]{Example}
\newtheorem{Cor}[Def]{Corollary}
\newtheorem{Thm}[Def]{Theorem}
\newtheorem{Prop}[Def]{Proposition}
\newtheorem{Lem}[Def]{Lemma}

\newtheorem*{Prop*}{Proposition}
\newtheorem*{Lem*}{Lemma}

\newtheorem{QN}{Question}

\makeatletter
\newtheorem*{rep@theorem}{\rep@title}
\newcommand{\newreptheorem}[2]{%
	\newenvironment{rep#1}[1]{%
		\def\rep@title{#2 \ref{##1}}%
		\begin{rep@theorem}}%
		{\end{rep@theorem}}}
\makeatother

\newreptheorem{Thm}{Theorem}
\newreptheorem{Cor}{Corollary}

\newcommand{\suchthat}{\;\ifnum\currentgrouptype=16 \middle\fi|\;}

\newcommand{\compcent}[1]{\vcenter{\hbox{$#1\circ$}}}

\newcommand{\comp}{\mathbin{\mathchoice
		{\compcent\scriptstyle}{\compcent\scriptstyle}
		{\compcent\scriptscriptstyle}{\compcent\scriptscriptstyle}}}
	
\newcommand\restr[2]{{
			\left.\kern-\nulldelimiterspace 
			#1 
			\vphantom{\big|} 
			\right|_{#2} 
}}

\begin{document}

\maketitle

\begin{abstract}
This paper studies how many orthogonal bi-invariant complex structures exist on a metric Lie algebra over the real numbers. Recently, it was shown that irreducible Lie algebras which are additionally $2$-step nilpotent admit at most one orthogonal bi-invariant complex structure up to sign. The main result generalizes this statement to metric Lie algebras with any number of irreducible factors and which are not necessarily $2$-step nilpotent. It states that there are either $0$ or $2^k$ such complex structures, with $k$ the number of irreducible factors of the metric Lie algebra. The motivation for this problem comes from differential geometry, for instance to construct non-parallel Killing-Yano $2$-forms on nilmanifolds or to describe the compact Chern-flat quasi-K\"ahler manifolds.

The main tool we develop is the unique orthogonal decomposition into irreducible factors for metric Lie algebras with no non-trivial abelian factor. This is a generalization of a recent result which only deals with nilpotent Lie algebras over the real numbers. Not only do we apply this fact to describe the orthogonal bi-invariant complex structures on a given metric Lie algebra, but it also gives us a method to study different inner products on a given Lie algebra, computing the number of irreducible factors and orthogonal bi-invariant complex structures for varying inner products.
\end{abstract}

\section{Introduction}

In differential geometry, a natural class of examples is provided by homogeneous manifolds, which are defined as Riemannian manifolds having a transitive action by isometries. Even the special case of Lie groups $G$ equipped with a left-invariant metric gives a rich geometry with many open questions. The geometry of these spaces is completely described by the tangent space at the identity element, which forms a metric Lie algebra, namely a Lie algebra $\lie$ equipped with an inner product $\langle \cdot, \cdot \rangle: \lie \times \lie \to \R$. A left-invariant almost complex structure on the manifold $G$ is then equivalent to a linear map $J: \lie \to \lie$ on the metric Lie algebra satisfying $J^2 = - \I_\lie$. 

This paper studies the almost complex structures on the metric Lie algebra $\lie$ which make the corresponding Lie group $G$ into a complex Lie group with a left-invariant metric. The $J$ satisfying this condition are called orthogonal bi-invariant complex structures, for which we will give more background in Section \ref{sec:notation}. The main question we study is how to describe all the different orthogonal bi-invariant complex structures on a given metric Lie algebra $\lie$.
\begin{QN}
\label{mainquestion}
	Given a metric Lie algebra $\lie$, how many orthogonal bi-invariant complex structures $J: \lie \to \lie$ does $\lie$ admit? Is there a method to describe them? 
\end{QN}
\noindent As explained before, given an orthogonal bi-invariant complex structure $J$, the real Lie algebra $\lie$ can be made into a complex Lie algebra with a Hermitian inner product. Vice versa, starting from a complex Lie algebra with a Hermitian inner product, we can take the underlying real Lie algebra and the real part of the inner product to find a real metric Lie algebra with an orthogonal bi-invariant complex structure. Hence Question \ref{mainquestion} is equivalent to studying how many metric complex Lie algebras have an isometric underlying real Lie algebra. In this way, Question \ref{mainquestion} can be considered as the metric counterpart of \cite[Question 5]{dlv12-1}, which asked whether two different complex Lie algebras could have an isomorphic underlying real Lie algebra. Examples of such complex Lie algebras, including a general way of describing them, was provided in \cite{dere19-2}. 

The motivation for Question \ref{mainquestion} comes from several applications in differential geometry, for which more details can be found in the references given below. In \cite{dlv12-1} it was shown that bi-invariant complex structures are closely related to Hermitian manifolds which are Chern-flat and quasi-K\"ahler. In fact, every compact manifold satisfying these conditions is isometric to the quotient of a complex $2$-step nilpotent Lie group $G$ by a cocompact lattice $\Gamma$. The almost-complex structure on $\Gamma \backslash G$ is not equal to the complex multiplication on $G$, but can be constructed from it by introducing a minus sign on the center, see \cite[Section 4.2.]{dlv12-1}. Hence, describing the number of orthogonal bi-invariant complex structures on a $2$-step nilpotent Lie algebra is equivalent to describing the number of different ways it can be made into a Chern-flat and quasi-K\"ahler manifold. 

A second application for orthogonal bi-invariant complex structures lies in the construction of non-trivial Killing-Yano $2$-forms as introduced in \cite{yano59-1}. These can be considered as the generalization of Killing vector fields since Killing-Yano $1$-forms are the dual of Killing vector fields. They play an important role in physics, namely as a condition for vacuum solutions of Einstein's field equations, see \cite{pw70-1}. Any parallel $p$-form is Killing-Yano, thus research focuses on constructing examples which are not parallel. On $2$-step nilpotent Lie groups, there are no left-invariant non-degenerate parallel $2$-forms by \cite[Theorem 5.1.]{ad18-1}, hence these are natural candidates for finding such examples. In both \cite{ad19-1,dbm19-1} independently, a bijection is constructed between Killing-Yano $2$-forms up to multiplication by a positive real number and orthogonal bi-invariant complex structures on a $2$-step nilpotent Lie algebra. The question how many orthogonal bi-invariant complex structures exist translates to the dimension of the space of Killing-Yano $2$-forms on a $2$-step nilpotent Lie algebra. 

On abelian Lie algebras of dimension $2n \geq 4$, there exists a continuum of different orthogonal bi-invariant structures, see Example \ref{ex:uncountable}. Therefore we will always assume that our Lie algebra $\lie$ has no non-zero abelian factor. Under that assumption, we show that any metric Lie algebra has finitely many orthogonal bi-invariant complex structures.

\begin{repThm}{thm:main}
Assume that a Lie algebra $\lie$ with no non-zero abelian factors has an orthogonal bi-invariant complex structure $J: \lie \to \lie$ and that the corresponding complex Lie algebra has a metric decomposition $$(\lie,J) = \displaystyle\bigoplus_{j=1}^k \left(\lie_j,J_j\right) $$ into irreducible components $\left(\lie_j,J_j\right)$. Every orthogonal bi-invariant complex structure on $\lie = \displaystyle\bigoplus_{j=1}^k \lie_j$ is of the form $\pm J_1 \oplus \ldots \oplus \pm J_k$. 
\end{repThm}
\noindent As we will show in Theorem \ref{thm:irreducible} the $k$ of this theorem, which is the number of irreducible factors of $(\lie,J)$, is equal to the number of irreducible factors of $\lie$ as a real metric Lie algebra, so it can be computed without having information about a complex structure $J$.

In the special case of irreducible Lie algebras which are $2$-step nilpotent, Theorem \ref{thm:main} was given in \cite[Proposition 4.9.]{dbm19-1}. For Lie algebras without inner product, \cite{dere19-2} shows that any Lie algebra has at most finitely many bi-invariant complex structures up to isomorphism and gives an explicit way to describe all of them, similarly as in Theorem \ref{thm:main}. The main tool in that proof was to consider so-called conjugate Lie algebras for elements of the Galois group of a field extension. Since the primary focus of this paper is the real case, there is only one non-trivial element in the Galois group, such that the conjugate Lie algebra corresponds to replacing a complex structure $J$ by $-J$.

As is clear from the statement, the main tool we will develop is the orthogonal decomposition of a metric Lie algebra into irreducible factors. For Lie algebras without metric, a similar result was given in \cite{fgh13-1}, where the indecomposable factors were unique up to isomorphism. The result for metric Lie algebras is stronger, in the sense that the irreducible factors are unique ideals of the Lie algebra.

\begin{repThm}{thm:decomposition}
Let $\lie$ be a metric Lie algebra with no non-zero abelian factor. There exist unique irreducible factors $\lie_j \subset \lie$ such that $$\lie = \bigoplus_{j=1}^k \lie_j$$ is written as an orthogonal decomposition of the ideals $\lie_j$. 
\end{repThm}

Both main results raise the question how the choice of inner product on the Lie algebra $\lie$ influences both the number of irreducible factors and the number of orthogonal bi-invariant complex structures. Section \ref{sec:app} will give some insight in how the inner product and algebraic structure of the Lie algebra are related.


\paragraph*{Acknowledgements}

I would like to thank Adri\'an Andrada for introducing me to Killing-Yano forms during my research stay at Universidad Nacional de C\'ordoba, which lead to this work.

\section{Preliminaries}
\label{sec:prel}

This section introduces the definitions and notations for the remainder of the paper. We start by recalling the basic notions of metric Lie algebras, afterwards introduce almost complex structures on real Lie algebras, including the relation to complex Lie algebras, and finally discuss the complexification or a real Lie algebra. Every vector space and Lie algebra we consider is assumed to be finite-dimensional.

\paragraph{Metric Lie algebras} 

\

A metric Lie algebra is a Lie algebra $\lie$ over some subfield $F \subset \C$, equipped with a positive definite Hermitian form $$\langle  \hspace{1mm}\cdot\hspace{1mm}, \hspace{1mm}\cdot \hspace{1mm}\rangle: \lie \times \lie \to F,$$ which we call the \textit{inner product} on $\lie$. In this paper, linearity of the Hermitian form is taken in the first component, the antilinearity in the second component. If the field $F$ consists of real numbers only, so $F \subset \R$, the inner product $\langle \cdot , \cdot \rangle$ is in fact a symmetric bilinear form. If the Lie algebra $\lie$ is abelian, meaning that $\lie$ is just a vector space, this corresponds to the regular notion of inner products on a vector space. Note that the space of inner products on a real vector space is a manifold of dimension $\frac{n (n+1)}{2}$ where $n$ is the dimension of the vector space, whereas the space of inner products on a complex vector space of dimension $n$ is a manifold of dimension $n^2$. 

The geometric importance of metric Lie algebras lies in the special case $F = \R$. Indeed, if $G$ is a Lie group with Lie algebra $\lie$, then the inner product $\langle \cdot, \cdot \rangle$ corresponds to a left-invariant Riemannian metric on $G$. This makes it possible to talk about geometric properties on a metric Lie algebra, such as the Riemannian curvature tensor or the Ricci curvature, see \cite{miln76-1}. But also the case $F = \C$ is crucial, because of its relation to bi-invariant complex structures on the underlying real Lie algebra. Hence we formulate some of the results in this paper, in particular Theorem \ref{thm:decomposition}, for general subfields of $\C$, comparable to \cite{dere19-2}.

Let $U$ be any vector space over the field $F$. If $f: U \to U$ is a linear map, we say that a subspace $V \subset U$ is $f$-invariant if $f(V) \subset V$. If $V, \hspace{1mm} W \subset U$ are subspace with $V + W = U$ and $V \cap W = 0$, then we write $U = V \oplus W$, representing thus the internal direct sum of subspaces of $U$. In the case that $U$ is equipped with an inner product, we additionally assume that $V$ and $W$ are orthogonal. In this case, we write $V = W^\perp$ and call $V$ the orthogonal complement of $W$. If $U = V \oplus W$ and $f_1: V \to V$ and $f_2: W \to W$ are linear maps, we will write $f_1 \oplus f_2: U \to U$ for the unique linear map on $U$ which is an extension of $f_1$ and $f_2$ to the whole vector space. Given metric Lie algebras $\lie$ and $\lieh$, the direct sum $\lie \oplus \lieh$ has a unique inner product such that $\lie$ and $\lieh$ are orthogonal. If no other inner product is specified, we will always assume that the inner product satisfies this assumption. We call a linear map $f: U \to U$ symmetric if $\langle f(X), Y \rangle = \langle X , f(Y) \rangle$ and skew-symmetric if $\langle f(X), Y \rangle = - \langle X , f(Y) \rangle$ for all $X, Y \in U$. Since any Lie algebra is in particular also a vector space, we can use the same notations as above in the context of metric Lie algebras. 

Let $\lie$ be any metric Lie algebra. Note that in general for a subalgebra $\lieh \subset \lie$, it does not hold that $\lieh^\perp$ is also subalgebra of $\lie$, even not if we additionally assume $\lieh$ to be an ideal. In the special case where both $\lieh$ and $\lieh^\perp$ are ideals, the internal direct sum $\lie = \lieh \oplus \lieh^\perp$ is an orthogonal direct sum of Lie algebras. Both ideals $\lieh$ and $\lieh^\perp$ are called factors of the metric Lie algebra $\lie$.
\begin{Def}
	We say that a non-zero Lie algebra $\lie$ is \textbf{irreducible} if for every factor $\lieh$ it holds that either $\lieh = 0$ or $\lieh = \lie$.
\end{Def}
\noindent Since our Lie algebras are assumed to be finite-dimensional, every Lie algebra has an orthogonal decomposition of the form $$ \lie = \bigoplus_{j=1}^k \lie_j$$ where every $\lie_j \subset \lie$ is an irreducible factor of $\lie$. The main result of Section \ref{sec:decomp} is that the irreducible factors $\lie_j$ are unique if $\lie$ has no non-zero abelian factor.

\paragraph{Complex structures}

\

The relation between real and complex Lie algebras is given by the notion of (almost) complex structures on Lie algebras.
\begin{Def}
	Let $\lie$ be any real Lie algebra and $J: \lie \to \lie$ a linear map.
	\begin{itemize}
		\item The map $J$ is called an \textbf{almost complex structure} on $\lie$ if $J^2 = - \I_\lie$.
		\item We call $J$ a \textbf{bi-invariant} complex structure if $$J \left( [X,Y]\right) = [J(X),Y]$$ for all $X, Y \in \lie$.
		\item If $\lie$ is a metric Lie algebra with inner product $\langle \cdot, \cdot \rangle$, we call $J$ \textbf{orthogonal} if $$\langle J(X), J(Y) \rangle = \langle X, Y \rangle$$ for all $X, Y \in \lie$, so if $J$ is an isometry for the inner product. The inner product $\langle \cdot, \cdot \rangle$ is called \textbf{Hermitian} with respect to the almost complex structure $J$.
	\end{itemize}
\end{Def}

\noindent The existence of an almost complex structure on a real Lie algebra implies that the Lie algebra has even dimension.  

Assume that $\lieh$ is a complex Lie algebra. By restricting scalar multiplication to the real numbers $\R$, we get a real Lie algebra $\lie$, for which it holds that $\dim_\R(\lie) = 2 \dim_\C(\lieh)$ and we call $\lie$ the (real) underlying Lie algebra of $\lieh$. In this case, we have a linear map $J: \lie \to \lie$ given by $J(X) = i X$ for $X \in \lie$. The map $J$ is a bi-invariant complex structure on $\lie$. Vice versa, given a real Lie algebra $\lie$ with a bi-invariant complex structure $J$, we can make it into a complex Lie algebra $\lieh$, by defining complex multiplication as $$(a + bi) X = a X + b J(X)$$ for all $a, b \in \R$. Since the focus lies on the underlying Lie algebra $\lie$ and the complex structure $J$, we will write $(\lie,J) = \lieh$ for complex Lie algebras.

Moreover, if $\lie$ is a metric Lie algebra and $J$ an orthogonal bi-invariant complex structure, we can equip the complex Lie algebra $(\lie,J)$ with an inner product $\langle \cdot, \cdot \rangle_\C$ defined as $$\langle X , Y \rangle_\C = \frac{\langle X , Y \rangle + i \langle X, J Y \rangle}{2}.$$ \label{definitionofinneronC} The dependency of $\langle \cdot, \cdot \rangle_\C$ on $J$ will be important in the proof of Proposition \ref{prop:uniquedecomposition}. To recover the inner product on $\lie$ we have the relation \begin{align}\label{eq:realpart}\langle X, Y \rangle = \langle X, Y \rangle_\C + \overline{\langle X, Y \rangle_\C}
\end{align} where $\overline{\lambda}$ is the complex conjugate of $\lambda \in \C$. 

More general, Equation (\ref{eq:realpart}) above defines an inner product $\langle \cdot, \cdot \rangle$ on the underlying real Lie algebra $\lie$ for every complex Lie algebra $\lieh$ with inner product $\langle \cdot, \cdot \rangle_\C$, for which the corresponding complex structure $J(x) = ix$ is orthogonal. Note that sometimes in literature, see for example \cite{huyb05-1}, the inner product $\langle \cdot,\cdot\rangle$ is defined without the factor $2$ in the denominator. We have chosen this definition since it will be convenient in the proof of Proposition \ref{prop:uniquedecomposition}. Since both metrics are conformally equivalent (meaning that they are the same up to scalar multiplication) this does not have any influence on the geometric properties of $(\lie,J)$ such as its orthogonal decomposition. Because the metric structures on the real and complex Lie algebra are equivalent, we can go back and forth without further mentioning the different metrics. 

\paragraph{Complexification}

\

If $\lie$ is a real Lie algebra, we can make it into a complex Lie algebra $\lie^\C$ by extending the scalars, namely taking the tensor product $\lie^\C= \lie \otimes_\R \C$. Every linear map $f: \lie \to \lie$ extends to a $\C$-linear map on $\lie^\C$ which we denote as $f^\C : \lie^\C \to \lie^\C$. If $f: \lie \to \lie$ is a morphism of real Lie algebras, then also $f^\C: \lie^\C \to \lie^\C$ will be a morphism of complex Lie algebras. The complex dimension of the complexification is equal to the real dimension of the original Lie algebra, so $\dim_\C(\lie^\C) = \dim_\R(\lie)$.

Any inner product $\langle \cdot, \cdot \rangle$ on $\lie$ extends to an inner product $\langle \cdot, \cdot \rangle^\C$ on $\lie^\C$ in the following way. Take any orthonormal basis $X_1, \ldots, X_n$ for the real Lie algebra $\lie$, then this is also a basis for the complexification $\lie^\C$. We now define the new inner product by taking $X_1, \ldots, X_n$ as an orthonormal basis of $\lie^\C$. More explicitly, the inner product is defined as 
\begin{align*}
\langle Y, Z \rangle^\C =\sum_{j=1}^n y_j \overline{z}_j
\end{align*}
for all $Y = \displaystyle \sum_{j=1}^n y_j X_j, Z = \displaystyle \sum_{j=1}^n z_j X_j \in \lie^\C$. For all elements $Y, Z \in \lie$ it holds that $\langle Y, Z \rangle^\C = \langle Y, Z \rangle$. Sometimes we will write $\langle \cdot,\cdot \rangle^\C$ without superscript ${}^\C$ when no confusion is possible.

\label{sec:notation}

\section{Decomposition into irreducible factors}
\label{sec:decomp}

This section discusses the decomposition of a metric Lie algebra into irreducible factors, as introduced in Section \ref{sec:notation}. In the case of Lie algebras without inner product \cite[Theorem 3.3.]{fgh13-1} shows that the indecomposable factors of a decomposition are unique up to isomorphism. However, \cite[Example 2.3.]{dere19-2} demonstrates that these isomorphic factors can be different ideals of the Lie algebra and thus two decompositions are not necessarily identical. The main result of this section shows that a stronger property holds for metric Lie algebras, namely that the irreducible factors of an orthogonal decomposition are unique ideals of the Lie algebra. 
\begin{Thm}
	\label{thm:decomposition}
	Let $\lie$ be a metric Lie algebra with no non-zero abelian factor. There exist a unique irreducible factors $\lie_j \subset \lie$ such that $$\lie = \bigoplus_{j=1}^k \lie_j$$ is written as an orthogonal decomposition of the ideals $\lie_j$. 
\end{Thm}
\noindent This will form one of the main ingredients for the main result describing the possible orthogonal bi-invariant complex structures on a metric Lie algebra.

Partial results for this decomposition were already given in \cite{dbm20-2} for the special case of nilpotent Lie algebras over $\R$, generalizing a previous result which only dealt with real $2$-step nilpotent Lie algebra in \cite{dbm20-1}. In these papers, the authors use the de Rham decomposition of a Lie group with a left-invariant metric and show that this induces a decomposition of the nilpotent Lie algebra into irreducible factors. As indicated in \cite[Example 2.3.]{dbm20-2} this method does not work for general Lie algebras, since the de Rahm decomposition does not lead to ideals of the corresponding Lie algebra. The proof we present here is algebraic in nature, by studying properties of orthogonal projections of the metric Lie algebra. Hence it works without any conditions on the field or the Lie algebra, except for the condition about no non-zero abelian factors.

On abelian metric Lie algebras, or equivalently vector spaces with an inner product, there exist many orthogonal bases by the Gram-Schmidt process. Each of these bases induces an orthogonal decomposition of the vector space into irreducible subspaces, implying that in the abelian case a decomposition into irreducible factors is far from unique. Also when the Lie algebra has an abelian factor of dimension $\geq 2$, uniqueness is not possible due to this observation. 

Hence in the remaining part of this section we assume that $\lie$ is a metric Lie algebra with no non-zero abelian factor. The following characterization of such Lie algebras is well-known, but for completeness we present the short proof.
\begin{Lem}
	\label{lem:noabelian}
	Let $\lie$ be a metric Lie algebra, then $\lie$ does not have a non-zero abelian factor if and only if $Z(\lie) \subset [\lie,\lie]$.
\end{Lem}
Here $Z(\lie)$ is the center of the Lie algebra $\lie$, given by $$Z(\lie) = \left\{ X \in \lie \suchthat \forall \hspace{1mm} Y \in \lie: [X,Y] = 0\right\}.$$ Note that the last condition of Lemma \ref{lem:noabelian} does not depend on the inner product on $\lie$.
\begin{proof}
	First assume $\lie$ has an abelian factor $\mathfrak{a} \neq 0$, so we can write $\lie = \lieh \oplus \mathfrak{a}$. In this case, $[\lie,\lie] = [\lieh,\lieh] \subset \lieh$ and $[\mathfrak{a},\lie] = [\mathfrak{a},\mathfrak{a}] = 0$, showing that $\mathfrak{a}$ is a subspace of $Z(\lie)$ which does not lie in $[\lie,\lie]$. This gives the first implication of the lemma.
	
	For the other direction, consider the subspace $$\mathfrak{a} = \{X \in Z(\lie) \suchthat \forall \hspace{1mm} Y \in Z(\lie) \cap [\lie,\lie]: \langle X, Y\rangle = 0 \} \subset Z(\lie),$$ which is exactly the orthogonal complement of $Z(\lie) \cap [\lie,\lie]$ in $Z(\lie)$. Note that $\mathfrak{a}$ is an abelian ideal of $\lie$ and the orthogonal complement contains $[\lie,\lie]$ and hence is also an ideal. We conclude that if $\mathfrak{a}$ is non-zero, or equivalently if $Z(\lie) \cap [\lie,\lie] \neq Z(\lie)$, the Lie algebra has a non-zero abelian factor. This finishes the other implication of the lemma. 
\end{proof}

The main technique for studying factors of a metric Lie algebra is to shift attention to the corresponding orthogonal projections. We first give the definition of an orthogonal projection of a metric Lie algebra. 

\begin{Def}
	\label{def:projection}
	We call a morphism of Lie algebras $p: \lie \to \lie$ an \textbf{orthogonal projection onto $p(\lie)$} if $p \comp p = p$ and for all $X, Y \in \lie$ it holds that
	\begin{enumerate}[$(i)$]
	\item $p([X,Y]) = [p(X),Y]$ and
	\item $\langle p(X),Y \rangle = \langle X, p(Y) \rangle$, i.e.~$p$ is symmetric.
	\end{enumerate}
\end{Def}

\noindent By the anticommutativity of the Lie bracket, condition ($i$) is equivalent to $$p([X,Y]) = [X,p(Y)].$$ We first show that every projection onto a factor satisfies these conditions.
\begin{Ex}
	\label{ex:orthogonalprojection}
Let $\lie$ be a metric Lie algebra with an orthogonal decomposition $\lie = \lie_1 \oplus \lie_2$. Consider the map $p: \lie \to \lie$ given by
\begin{align*}
p(X_1 + X_2) = X_1
\end{align*} for all $X_1 \in \lie_1, \hspace{1mm} X_2 \in \lie_2$. Clearly $p$ is a morphism of Lie algebras since $\lie_2$ is an ideal and the definition shows that $p \comp p = p$. For every $X_1, Y_1 \in \lie_1, \hspace{1mm} X_2, Y_2 \in \lie_2$ we have that
$$p([X_1 + X_2, Y_1 + Y_2]) = p([X_1,Y_1] + [X_2,Y_2]) = [X_1,Y_1] = [X_1,Y_1 + Y_2] = [p(X_1+X_2),Y_1 + Y_2]$$ and so the first condition of Definition \ref{def:projection} holds. A similar computation implies that also the second condition holds and thus $p$ is an orthogonal projection onto $\lie_1$.
\end{Ex}

The following lemma shows that the only orthogonal projections are the ones introduced in Example \ref{ex:orthogonalprojection}, so every orthogonal projection is onto a factor of the Lie algebra.

\begin{Lem}
Let $p: \lie \to \lie$ be an orthogonal projection of a metric Lie algebra $\lie$, then we have an orthogonal decomposition $$\lie = p(\lie) \oplus \ker(p).$$
\end{Lem}
\noindent The approach for the rest of this section will be to study orthogonal projections $p: \lie \to \lie$ as a substitute for factors of the metric Lie algebra $\lie$.

\begin{proof}
It is clear that $\ker(p)$ is an ideal of $\lie$ and that $\dim(p(\lie)) + \dim(\ker(p))= \dim(\lie)$ by the dimension theorem for linear maps. Hence it suffices to show that $p(\lie)$ is an ideal which is orthogonal to $\ker(p)$. 

For the first statement, assume that $X \in \lie$ and $p(Y) \in p(\lie)$, then $$[X,p(Y)] = p([X,Y]) \in p(\lie),$$ so $p(\lie)$ is an ideal. For the second statement, assume that $X \in \ker(p)$ and $p(Y) \in p(\lie)$, then $$\langle X, p(Y) \rangle = \langle p(X), Y \rangle = \langle 0, Y \rangle = 0.$$ This gives us the statement of the lemma.\end{proof}

By applying this result to irreducible metric Lie algebras, we see that there are only two posibilities for orthogonal projections.
\begin{Cor}
	\label{cor:irreducibleprojection}
If $\lie$ is an irreducible metric Lie algebra, then it only has two orthogonal projections $p: \lie \to \lie$, namely either $p(X) = 0$ or $p(X) = X$ for all $X \in \lie$.
\end{Cor}
To apply this corollary, we will need that the restriction of an orthogonal projection is again an orthogonal projection.

\begin{Lem}
\label{lem:restriction}
Let $p: \lie \to \lie$ be an orthogonal projection on a metric Lie algebra $\lie$ such that $p(\lie) \subset \lieh$ for a subalgebra $\lieh \subset \lie$. The restriction $\restr{p}{\lieh}: \lieh \to \lieh$ is an orthogonal projection of $\lieh$.
\end{Lem}
\begin{proof}
Note that $\restr{p}{\lieh}$ is indeed well-defined by the assumption on $p$. The conditions follow immediately from the fact that $p$ is an orthogonal projection.
\end{proof}

The main property for proving Theorem \ref{thm:decomposition} is that two orthogonal projections of a metric Lie algebra with no non-zero abelian factor always commute. We first formulate a more general proposition, which will imply as well that two orthogonal bi-invariant complex structures commute in the next section, see Corollary \ref{prop:commute}.

\begin{Prop}
	\label{prop:commutegeneral}
	Let $\lie$ be a metric Lie algebra with no non-zero abelian factor. Assume that $f_1: \lie \to \lie$ and $f_2: \lie \to \lie$ are two linear maps satisfying $$[f_j(X),Y] = f_j([X,Y])$$ for $1 \leq j \leq 2$ and for all $X, Y \in \lie$. If the maps $f_1$ and $f_2$ are symmetric or skew-symmetric, then the maps $f_1$ and $f_2$ commute. 
\end{Prop}
\begin{proof}
	Note that the assumption implies that the commutator subalgebra $[\lie,\lie]$ is invariant under $f_1$ and $f_2$. Just as for the orthogonal projections, we have that $$f_j([X,Y]) = [X,f_j(Y)]$$ for all $X, Y \in \lie$ by the anticommutativity of the Lie bracket. Write $V = [\lie,\lie]^\perp$, then for every $X \in V$ and $Y \in [\lie,\lie]$, it holds that $$\langle f_j(X), Y \rangle =  \pm \langle X , f_j(Y) \rangle = 0$$ and hence $V$ is also invariant under $f_1$ and $f_2$. It suffices to show that $f_1$ and $f_2$ commute for elements in both $[\lie,\lie]$ and $V$.
	
	First consider $[X,Y] \in [\lie,\lie]$ with $X, Y \in \lie$, then $$f_1\left(f_2\left([X,Y]\right)\right) = f_1\left([f_2\left(X\right),Y]\right) = [f_2\left(X\right),f_1\left(Y\right)] = f_2\left([X,f_1\left(Y\right)]\right) = f_2\left(f_1\left([X,Y]\right)\right)$$ and thus, since the elements $[X,Y]$ span $[\lie,\lie]$, $f_1$ and $f_2$ commute on $[\lie,\lie]$. 
	
	For the subspace $V$ we prove this by contradiction. Assume that there is an $X \in V$ such that $f_1\left(f_2\left(X\right)\right) \neq f_2\left(f_1\left(X\right)\right)$, then we have $0 \neq f_1\left(f_2\left(X\right)\right) - f_2\left(f_1\left(X\right)\right) \in V$. So, since $\lie$ has no non-zero abelian factor, Lemma \ref{lem:noabelian} shows that $V \cap Z(\lie) = 0$ and thus that there exists $Y \in \lie$ with $[f_1\left(f_2\left(X\right)\right) - f_2\left(f_1\left(X\right)\right),Y] \neq 0$. This is a contradiction, since $$[f_1\left(f_2\left(X\right)\right) - f_2\left(f_1\left(X\right)\right),Y] = f_1\left(f_2\left([X,Y]\right)\right) - f_2\left(f_1\left([X,Y]\right)\right) = 0$$ by the first case. This shows that $f_1$ and $f_2$ commute on $V$ and hence on $\lie$.
\end{proof}
\noindent We need that the maps $f_1$ and $f_2$ are (skew-)symmetric in order to show that the subspace $V$ is invariant. If this assumption does not hold, one can only show that $\left(f_1 \comp f_2 - f_2 \comp f_1\right)(\lie) \subset Z(\lie)$, see Example \ref{ex:notcommute} in the next section for an example in the case of bi-invariant complex structures. 

\begin{Cor}
	\label{prop:projcommute}
	Let $p_1, p_2: \lie \to \lie$ be two orthogonal projections of a metric Lie algebra $\lie$ with no non-zero abelian factor. The maps $p_1$ and $p_2$ commute and $p_1 \comp p_2 = p_2 \comp p_1$ is an orthogonal projection on $p_1(\lie) \cap p_2(\lie)$.
\end{Cor}

\begin{proof}
The first statement follows immediately from Proposition \ref{prop:commutegeneral} and the fact that $p_1$ and $p_2$ are symmetric. For the last statement, write $p_0 = p_1 \comp p_2 = p_2\comp p_1$, then it is clear that $p_0$ is a morphism of Lie algebras and that $p_0 \comp p_0 = p_0$ since $p_1$ and $p_2$ commute. The other two conditions of Definition \ref{def:projection} follow immediately from a computation. Therefore it is left to show that the image $$p_0(\lie) = p_1(\lie) \cap p(\lie_2).$$ For this, note that the inclusion $p_0(\lie) = p_1(p_2(\lie)) = p_2(p_1(\lie))\subset p_1(\lie) \cap p_2(\lie)$ follows directly from the definition of $p_0$. For the other inclusion, assume $X \in p_1(\lie) \cap p_2(\lie)$. From $p_1 \comp p_1 = p_1$ we conclude that $p_1(X) = X$ and similarly also $p_2(X) = X$, which leads to $p_0(X)=X$ and thus $X \in p_0(\lie)$.
\end{proof}

By applying the previous results, we get the main result of this section, namely the uniqueness of the orthogonal decomposition into irreducible factors.

\begin{proof}[Proof of Theorem \ref{thm:decomposition}]
There always exist irreducible factors $\lie_j$ to write $\lie$ as an orthogonal direct sum $\lie = \displaystyle \bigoplus_{j=1}^k \lie_j$, so it suffices to show that any irreducble factor is equal to some $\lie_j$. Consider the orthogonal projections $p_j: \lie \to \lie$ onto $\lie_j = p_j(\lie)$ corresponding to the decomposition into irreducible factors, see Example \ref{ex:orthogonalprojection}. Let $q$ be any other orthogonal projection onto an irreducible factor $\lieh$, then we will show that $\lieh = \lie_j$ for some $1 \leq j \leq k$.

Consider the orthogonal projections $q_j = p_j \comp q$ onto $\lieh_j = \lieh \hspace{0.2mm} \cap \hspace{0.2mm} \lie_j$ from Proposition \ref{prop:projcommute}. The restrictions $$\restr{q_j}{\lie_j}: \lie_j \to \lie_j$$ are orthogonal projections of the irreducible Lie algebra $\lie_j$, so Corollary \ref{cor:irreducibleprojection} and Lemma \ref{lem:restriction} imply that either $\lieh_j = \lie_j$ or $\lieh_j = 0$. Since the sum $\sum\limits_{i=1}^k p_j = \I_\lie$ equals the identity map, we know that $\sum\limits_{i=1}^k q_j = q$ thus $$\lieh = q(\lie) = \sum_{i=1}^k q_j(\lie) = \lieh_1 \oplus \ldots \oplus \lieh_k.$$ Since $\lieh$ is irreducible, there exists $1 \leq j \leq k$ with $\lieh = \lieh_j$ and $\lieh_j = 0$ for $i \neq j$, implying $\lieh = \lie_j$ as we wanted.
\end{proof}

\section{Orthogonal bi-invariant complex structures}
\label{sec:main}
In this section we describe the orthogonal bi-invariant complex structures on general metric Lie algebras, showing that there are only finitely many on Lie algebras without irreducible abelian factors. The main tool is to consider the eigenspaces of the given bi-invariant complex structure in the complexification of the real Lie algebra. This means that we consider two different complex structures, one given by the original bi-invariant complex structure and a second one given on the complexification. The relation between both is the gist of the argument.

On the abelian Lie algebra $\R^2$ with the standard Euclidean metric, there exists up to sign only one orthogonal bi-invariant complex structure, given by the matrix $$J = \begin{pmatrix} 0 & -1 \\ 1 & 0 \end{pmatrix}.$$ In contrast, there are uncountably many different orthogonal bi-invariant complex structures on $\R^{2n}$ with $n > 1$.

\begin{Ex}
\label{ex:uncountable}
Consider for every $\lambda \in \R$ the orthogonal matrices $$J_\lambda = \frac{1}{\sqrt{\lambda^2 +1 }} \begin{pmatrix} 0 & 1 & -\lambda & 0 \\ -1 & 0 &0 & \lambda \\ \lambda & 0 &0 &1\\ 0 & -\lambda & -1 & 0\end{pmatrix}$$ for the standard inner product. A computation show that $J_\lambda^2 = -\I_4$, hence giving uncountably many different orthogonal bi-invariant complex structures on $\R^4$. This examples easily extend to $\R^{2n}$ for all $n > 1$. 
\end{Ex}
	\noindent Similarly, once a Lie algebra has an abelian factor of dimension $\geq 4$, it will have uncountable many different bi-invariant complex structures. Hence we will assume that $\lie$ has no non-zero abelian factors or equivalently, by Lemma \ref{lem:noabelian}, that $Z(\lie) \subset [\lie,\lie]$.

%
%

The main result of this section is a description of all bi-invariant complex structures on a real Lie algebra. If $\lie$ is a real metric Lie algebra with orthogonal bi-invariant complex structure, then by Theorem \ref{thm:decomposition} we can decompose the complex metric Lie algebra with metric $\langle \cdot, \cdot \rangle_\C$ as $$(\lie,J) = \displaystyle\bigoplus_{j=1}^k \lieh_j$$ with complex irreducible factors $\lieh_j$. Note that for every element $X, Y \in \lie$, if $\langle X, Y \rangle_\C =0$ then in particular $\langle X, Y \rangle =0$ holds as well. Hence the underlying real Lie algebras $\lie_j$ of $\lieh_j$ form factors of the real Lie algebra with metric $\langle \cdot, \cdot \rangle$, which are moreover invariant under the bi-invariant complex structure $J: \lie \to \lie$. This implies that if the Lie algebra $\lie$ is irreducible, then also $(\lie,J)$ is irreducible. The restriction of $J$ to $\lie_j$ is denoted as $J_j: \lie_j \to \lie_j$ and thus $\lieh_j = (\lie_j,J_j)$. These irreducible factors are all the information we need to determine the orthogonal bi-invariant complex structures on $\lie$.

\begin{Thm}
	\label{thm:main}
Assume that a Lie algebra $\lie$ with no non-zero abelian factor has an orthogonal bi-invariant complex structure $J: \lie \to \lie$ and that the corresponding complex Lie algebra has a metric decomposition $$(\lie,J) = \displaystyle\bigoplus_{j=1}^k \left(\lie_j,J_j\right) $$ into irreducible components as explained before the theorem. Every orthogonal bi-invariant complex structure on $\lie = \displaystyle\bigoplus_{j=1}^k \lie_j$ is of the form $\pm J_1 \oplus \ldots \oplus \pm J_k$. 
\end{Thm}

It is clear that all the maps $\pm J_1 \oplus \ldots \oplus \pm J_k$ are indeed orthogonal bi-invariant complex structures, so the true statement of the theorem is that there are no others. The main idea will be to consider the complexification of the Lie algebra and decompose it into the eigenspaces of the map $J$, which will form orthogonal factors of $\lie^\C$. Finally, using the fact that two different orthogonal bi-invariant complex structures always commute, we prove Theorem \ref{thm:main}.

\paragraph{The complexification $\lie^\C$}

\

First, we relate the complexification of a metric Lie algebra to the complex Lie algebra given by an orthogonal bi-invariant complex structure. 
\begin{Prop}
	\label{prop:uniquedecomposition}
	Let $\lie$ be a real metric Lie algebra with orthogonal bi-invariant complex structure $J: \lie \to \lie$. The complexification $\lie^\C$ is isometric as a Lie algebra to $(\lie,J) \oplus (\lie,-J)$.
\end{Prop}
\noindent Isometric as a Lie algebra means that there is a isomorphism of Lie algebras which also preserves the inner product. Recall that $(\lie,J)$ is equipped with the inner product described on page \pageref{definitionofinneronC}, depending on the map $J$.

\begin{proof}
	Write the inner product on $(\lie,J)$ as $\langle\cdot,\cdot\rangle_1$ and on $(\lie,-J)$ as $\langle \cdot,\cdot\rangle_2$. The inner product on the direct sum is denoted as $\langle \cdot,\cdot\rangle_\C$. We consider the injective map $$\varphi: \lie \to (\lie,J) \oplus (\lie,-J)$$ given by $\varphi(X) = (X,X)$, which preserves the Lie bracket. The map $\varphi$ preserves the inner product because 
	\begin{align*}\langle \varphi(X),\varphi(Y) \rangle_\C = \langle (X,X),(Y,Y) \rangle_\C &= \langle X, Y \rangle_1 + \langle X, Y \rangle_2 \\ &= \frac{\langle X, Y \rangle + i \langle X, J(Y) \rangle + \langle X, Y \rangle +i \langle X, -J(Y) \rangle  }{2} \\ &= \langle X, Y \rangle.\end{align*}
	
	Because $\lie^\C$ and $(\lie,J) \oplus (\lie,-J)$ have the same complex dimension, it suffices to show that $\varphi$ maps elements in $\lie$ which are linearly independent over $\R$ to elements in $(\lie,J) \oplus (\lie,-J)$ which are linearly independent over the complex numbers. For this, take $X_1, \ldots, X_n \in \lie$ which are linearly independent, and consider a linear combination $$0 = \sum_{j=1}^n \lambda_j \varphi(X_j) = \sum_{j=1}^n\left( \lambda_j X_j, \lambda_j X_j\right)$$ with $\lambda_j = a_j+ b_j i \in \C$. In the first component, this corresponds to $$\sum_{j=1}^n a_j X_j + b_j J(X_j) = 0,$$ in the second component to $$\sum_{j=1}^n a_j X_j - b_j J(X_j) = 0.$$ By adding and subtracting these two equations, using that $J$ is an isomorphism, we get that $2 \displaystyle \sum_{j=1}^n a_j X_j = 0 = 2 \displaystyle \sum_{j=1}^n b_j X_j$ and thus $a_j = b_j = 0$ for all $1 \leq j \leq n$. This gives us the result of the theorem.
\end{proof}

As a consequence, we can compute the number of irreducible factors of the complexification of a metric Lie algebra.
\begin{Cor}
	\label{cor:numberofirreducible}
	Let $\lie$ be a metric Lie algebra with an orthogonal bi-invariant complex structure $J$. If $(\lie,J)$ has $k$ irreducible factors, then $\lie^\C$ will have exactly $2k$ irreducible factors.	
\end{Cor}
\begin{proof}
It is clear that if the orthogonal decomposition of $(\lie,J)$ is given by $$(\lie,J) = (\lie_1,J_1) \oplus \ldots \oplus (\lie_k,J_k),$$ then the orthogonal decomposition of $(\lie,-J)$ is given by $$(\lie,-J) = (\lie_1,-J_1) \oplus \ldots \oplus (\lie_k,-J_k).$$ The statement hence follows from Theorem \ref{prop:uniquedecomposition}.
\end{proof}

\paragraph{Diagonalizing $J^\C$}

\

For the next step, we have to recall some notions of \cite{dere19-2}, which we only use in the special case of the extension $\R \subset \C$. If $V$ and $W$ are complex vector spaces, we call a map $\varphi: V \to W$ antilinear if $\varphi(\lambda X + \mu Y) = \overline{\lambda} \varphi(X) +  \overline{\mu} \varphi(Y)$ for all $X, Y \in V$. If $\lie$ and $\lieh$ are Lie algebras over $\C$, we call $\varphi: \lie \to \lieh$ an antilinear morphism if it is both antilinear and it preserves the Lie bracket, i.e.~$$\varphi([X,Y]) = [\varphi(X),\varphi(Y)].$$
If $\varphi$ is moreover a bijection, we will call it an antilinear isomorphism. 

The main example of antilinear isomorphisms comes from complex conjugation on the complexification $\lie^\C$. 

\begin{Ex}
\label{ex:antilinear}
Let $\lie$ be a real Lie algebra with complexification $\lie^\C$. Give a basis $X_1, \ldots, X_n$ for $\lie$ we can define an antilinear map $\sigma: \lie^\C \to \lie^\C$ as $$\sigma\left( \sum_{j=1}^n \lambda_j X_j \right) = \sum_{j=1}^n \overline{\lambda_j} X_j.$$ This map does not depend on the choice of the basis $X_j$ for $\lie$ and is a bijection on $\lie$. Obviously it respects the Lie bracket on elements in $\lie \subset \lie^\C$ since it is the identity map there. Now for every $Y = \displaystyle \sum_{j=1}^n \lambda_j X_j$ and $Z = \displaystyle \sum_{l=1}^n \mu_l X_l$ in $\lie^\C$, we have that $$\sigma([Y,Z]) = \sum_{j=1}^n \sum_{l=1}^n \overline{\lambda_j \mu_l} \sigma([X_j,X_l]) =\sum_{j=1}^n \sum_{l=1}^n \overline{\lambda_j \mu_l} [\sigma(X_j),\sigma(X_l)] = [\sigma(Y),\sigma(Z)]$$ and thus $\sigma$ preserves the bracket on $\lie^\C$. We conclude that $\sigma$ is an antilinear isomorphism of $\lie^\C$, which we will call the complex conjugation map on $\lie^\C$. Moreover, for every complex subalgebra $\lieh \subset \lie^\C$, it holds that $\sigma$ induces an antilinear isomorphism $\lieh \to \sigma(\lieh)$. In case $f^\C$ is the extension of a linear map $f: \lie \to \lie$, then $\sigma \comp f^\C = f^\C \comp \sigma$. 
 
Assume that $\lie$ is equipped with an inner product $\langle\cdot, \cdot\rangle$, which can be extended to $\langle \cdot,\cdot \rangle^\C$ on $\lie^\C$ as described in Section \ref{sec:prel}. A computation shows that for every $Y, Z \in \lie^\C$, it holds that $$\langle \sigma(Y),\sigma(Z)\rangle^\C= \overline{\langle Y, Z \rangle^\C}.$$ Therefore $\sigma$ maps any orthogonal decomposition of a subspace $\lieh$ to an orthogonal decomposition of $\sigma(\lieh)$ and hence $\lieh$ and $\sigma(\lieh)$ have the same number of irreducible factors.
\end{Ex}
Vice versa, we remark that every antilinear isomorphism can be described as in Example \ref{ex:antilinear}. Assume that $\varphi: \lie_1 \to \lie_2$ is an antilinear isomorphism of complex Lie algebras. Consider the complex Lie algebra $\lie_1 \oplus \lie_2$ which contains the subset $$\lie = \{(X,\varphi(X) \mid X \in \lie_1\}$$ considered as a real vector space. Note that $[\lie,\lie] \subset \lie$ because $\varphi$ preserves the Lie bracket, so $\lie$ is in fact a real Lie algebra. From \cite[Theorem 4.1.]{dere19-2} it follows that $\lie^\C = \lie_1 \oplus \lie_2$. The map $\sigma: \lie^\C \to \lie^\C$ given by $\sigma(X,Y) = (\varphi^{-1}(Y),\varphi(X))$ is an antilinear bijection which preserves the Lie bracket. Since $\restr{\sigma}{\lie} = \I_{\lie}$, it is thus equal to the complex conjugation map as in Example \ref{ex:antilinear} above. We conclude that every antilinear isomorphism $\varphi: \lie_1 \to \lie_2$ can be considered as complex conjugation in some complexification of a real Lie algebra. as in Example \ref{ex:antilinear}.

%
On the complexification $\lie^\C$, every semisimple map can be diagonalized. In the case of the complexification of an orthogonal bi-invariant complex structures, this leads to an orthogonal decomposition of the Lie algebra.
\begin{Prop}
	\label{prop:antilinear}
	Let $\lie$ be a metric real Lie algebra with an orthogonal bi-invariant complex structure $J$. The complexification $\lie^\C$ has an orthogonal decomposition $\lie^\C = \lie_1 \oplus \lie_{-1}$ which is $J^\C$-invariant and such that $\restr{J^\C}{\lie_1} = i \I_{\lie_1}$ and $\restr{J^\C}{\lie_{-1}} = - i \I_{\lie_{-1}}$. Moreover, complex conjugation on $\lie^\C$ gives an antilinear isomorphism between $\lie_1$ and $\lie_{-1}$. 
\end{Prop}

\begin{proof}
	The map $J$ has finite order, so it is semisimple and hence the complexification $J^\C$ is diagonalizable. Since $J^2 = -\I_{\lie}$, it can only have eigenvalues $i$ and $-i$. Write $\lie_1$ for the eigenspace for eigenvalue $i$ and $\lie_{-1}$ for the eigenspace for eigenvalue $-i$, then we have $\lie^\C = \lie_1 \oplus \lie_{-1}$ as a vector space. We have to show that this is an orthogonal decomposition of ideals and that complex conjugation gives an antilinear isomorphism. 
	
	Take $X \in \lie_1, Y \in \lie_{-1}$, then using that $J^\C$ is skew-symmetric, we get $$\langle X, Y \rangle^\C = -i \langle i X,Y \rangle^\C =  -i \langle J^\C(X),Y \rangle^\C = i \langle X, J^\C(Y) \rangle^\C = i \langle X, -i Y \rangle^\C   = i^2 \langle X, Y \rangle^\C = - \langle X, Y \rangle^\C$$  and hence $\langle X, Y \rangle^\C = 0$ showing that the decomposition is orthogonal. To show that $\lie_1$ is an ideal, we have for all $X \in \lie_1, Y \in \lie$ that $$J^\C([X,Y]) = [J^\C(X),Y] = i [X,Y]$$ and thus $[X,Y] \in \lie_1$. We conclude that $\lie_1$ is an ideal and exactly the same argument shows $\lie_{-1}$ is an ideal as well.
	
	Consider the complex conjugation map $\sigma: \lie^\C \to \lie^\C$ which commutes with $J^\C$. If $X \in \lie_1$, then $J^\C(X) = i X$ and thus $$J^\C(\sigma(X)) = \sigma(J^\C(X)) = \sigma(iX) = -i \sigma(X),$$ showing that $\sigma(X) \in \lie_{-1}$. We conclude that $\sigma(\lie_1) \subset \lie_{-1}$ and similarly $\sigma(\lie_{-1}) \subset \lie_1$. This gives the last statement of the proposition, since $\sigma$ is a bijection.
\end{proof}

\paragraph{Orthogonal bi-invariant complex structures commute}

\

To compare different decompositions into eigenspaces, we will use that different orthogonal bi-invariant complex structures always commute if there are no abelian factors.
\begin{Prop}
	\label{prop:commute}
	Let  $\lie$ be a metric real Lie algebra with no non-zero abelian factor. If $J_1, J_2: \lie \to \lie$ are two orthogonal bi-invariant complex structures, then the linear maps $J_1$ and $J_2$ commute. 
\end{Prop}
\begin{proof}
	This follows immediately from Proposition \ref{prop:commutegeneral} and the fact that orthogonal bi-invariant complex structures are skew-symmetric by definition.
\end{proof}

In the proof of Proposition \ref{prop:commute}, the assumption that $J_1$ and $J_2$ are orthogonal is crucial for the existence of an invariant subspace $V$ in the proof. If we omit this condition, we can only show that $\left(J_1\comp J_2 - J_2 \comp J_1\right)(\lie) \subset Z(\lie)$. We give a concrete example to illustrate the necessity of this assumption.

\begin{Ex}
	\label{ex:notcommute}
	Consider the $6$-dimensional nilpotent Lie algebra $\lie$ with basis $X_1, \ldots, X_6$ and brackets 
	\begin{align*}
	&[X_1,X_3] = X_5&
	&[X_1,X_4] = X_6&\\
	&[X_2,X_3] = X_6& 
	&[X_2,X_4] = -X_5&.
	\end{align*}
	A direct computation shows that the linear map $J_1: \lie \to \lie$ given by 
	\begin{align*}
	&J_1(X_1) = X_2& 
	&J_1(X_2) = -X_1&\\
	&J_1(X_3) = X_4&
	&J_1(X_4) = -X_3&\\
	&J_1(X_5) = X_6&
	&J_1(X_6) = - X_5&
	\end{align*} is a bi-invariant complex structure, making $(\lie,J_1)$ into the complex Heisenberg algebra $\lieh_3(\C)$. A second bi-invariant complex structure $J_2: \lie \to \lie$ is given by 
	\begin{align*}
	&J_2(X_1) = X_2 + X_6& 
	&J_2(X_2) = -X_1 + X_5&\\
	&J_2(X_3) = X_4&
	&J_2(X_4) = -X_3&\\
	&J_2(X_5) = X_6&
	&J_2(X_6) = - X_5&.
	\end{align*} Since $J_1(J_2(X_1)) = -X_1 - X_5$ and $J_2(J_1(X_1)) = -X_1 + X_5$, the linear maps $J_1$ and $J_2$ do not commute. By using Proposition \ref{prop:commute} we conclude that there does not exist a metric on $\lie$ which makes both $J_1$ and $J_2$ orthogonal.
\end{Ex}

With these tools we are ready for the proof of the main result in this section.

\begin{proof}[Proof of Theorem \ref{thm:main}]
As explained before, the maps of the form $\pm J_1 \oplus \ldots \oplus \pm J_k$ are indeed orthogonal bi-invariant complex structures on $\lie$. It suffices to show that there cannot be more that $2^k$ orthogonal bi-invariant complex structures.

Assume that $\tilde{J}$ is an orthogonal bi-invariant complex structure on $\lie$. By Proposition \ref{prop:commute} we know that $\tilde{J}$ and $J$ commute. Consider the complexification $\lie^\C$ and the decomposition $\lie^\C = \lie_1 \oplus \lie_{-1}$ given by Proposition \ref{prop:antilinear} for the map $J$. The number of irreducible factors of $\lie_1$ and $\lie_{-1}$ is the same and thus equal to $k$ by Corollary \ref{cor:numberofirreducible}. Since $\tilde{J}$ and $J$ commute, we know that $\tilde{J}^\C(\lie_1) = \lie_1$. Because $\sigma(\lie_1) = \lie_{-1}$ and $\sigma \tilde{J}^\C \sigma^{-1} = \tilde{J}^\C$, the map $\tilde{J}^\C$ is completely determined by its restriction to $\lie_1$, so it suffices to show that there are at most $2^k$ possibilities for $\restr{\tilde{J}^\C}{\lie_1}$.

The map $\tilde{J}^\C$ is semisimple and can only have eigenvalues $\pm i$. Just as in Proposition \ref{prop:antilinear}, the eigenspaces for eigenvalue $i$ and $-i$ give an orthogonal decomposition of $\lie_1$. Each of these eigenspaces consists of a number of irreducible factors of $\lie_1$ by the uniqueness of Theorem \ref{thm:decomposition}. Since the complex Lie algebra $\lie_1$ has exactly $k$ irreducible factors, there are exactly $2^k$ possibilities for such a decomposition into eigenspaces. Hence there are at most $2^k$ possibilities for the map $\tilde{J}^\C$ as well, giving the statement of the theorem.
\end{proof}

In contract to the abelian case, see Example \ref{ex:uncountable}, most metric Lie algebras have only finitely many orthogonal bi-invariant complex structures.
\begin{Cor}
	\label{cor:finitenumber}
	Any metric Lie algebra with no abelian irreducible factor has only a finite number of orthogonal bi-invariant complex structures.
\end{Cor}

\noindent Moreover, if the number is non-zero and given one orthogonal bi-invariant complex structure on a metric Lie algebra, the exact number is $2^k$ with $k$ the number of irreducible factors of the corresponding complex Lie algebra.

As another corollary, we get the generalization of \cite[Proposition 4.9.]{dbm19-1}  from $2$-step nilpotent Lie algebras to general Lie algebras.

\begin{Cor}
\label{cor:numberonirreducible}
If $\lie$ is an irreducible metric Lie algebra, then it has either zero or two orthogonal bi-invariant complex structures. In the latter case, the complex structures are equal up to sign.
\end{Cor}

\begin{proof}
If $\lie$ is abelian, then this is immediate since an irreducible abelian Lie algebra is isomorphic to $\R$ and hence has no complex structure. So assume that $\lie$ is irreducible and non-abelian with at least one orthogonal bi-invariant complex structure $J: \lie \to \lie$. In this case also $(\lie,J)$ is irreducible as a complex Lie algebra and has no non-zero abelian factor. The result now follows directly from Theorem \ref{thm:main}.
\end{proof}

If we examine Theorem \ref{thm:main}, the same conclusion already follows from the seemingly weaker condition that the Lie algebra $(\lie,J)$ is irreducible as a complex Lie algebra with $J$ any orthogonal bi-invariant complex structure on $\lie$. The next theorem shows that both assumptions are in fact equivalent, leading to an alternative approach for proving Theorem \ref{thm:main}.
\begin{Thm}
	\label{thm:irreducible}
	Let $\lie$ be a real metric Lie algebra with orthogonal decomposition $\lie =\displaystyle\bigoplus_{j=1}^k \lie_j$ into irreducible factors. For every orthogonal bi-invariant complex structure $J: \lie \to \lie$ it holds that $J(\lie_j) = \lie_j$, i.e.~the irreducible factors $\lie_j$ are $J$-invariant. Moreover, the ideals $\left(\lie_j,\restr{J}{\lie_j}\right)$ are the irreducible factors of $(\lie,J)$.
\end{Thm}
As an immediate consequence it follows that the orthogonal decompositions of the real metric Lie algebra $\lie$ and the complex Lie algebra $(\lie,J)$ are identical. The $k$ in Theorem \ref{thm:main} can hence be computed from the metric Lie algebra $\lie$.
\begin{proof}
	If $p_j: \lie \to \lie$ is the orthogonal projection onto $\lie_j$, then Proposition \ref{prop:commutegeneral} implies that $p_j$ and $J$ commute. In particular, $J$ maps the image $p_j(\lie) = \lie_j$ on itself, leading to the first statement. For the last statement of the corollary, we note that $\left(\lie_j,\restr{J}{\lie_j}\right)$ is irreducible since $\lie_j$ is irreducible and hence these are equal to the irreducible factors of $(\lie,J)$.
\end{proof}

This gives an alternative approach to Theorem \ref{thm:main} by first proving it for an irreducible Lie algebra and then using the theorem above to extend it to general metric Lie algebras.

\section{Varying the inner product on the Lie algebra}
\label{sec:app}
Although Theorem \ref{thm:main} gives a clear picture of the possibilities for orthogonal bi-invariant complex structures on metric Lie algebras, it does not tell us what happens if we change the inner product on $\lie$. As we will show in this section, the number of irreducible components for different metrics can take any value between $1$ and some upper bound determined by the algebraic structure. We also study how this influences the number of orthogonal bi-invariant complex structures. A full answer to this problem remains open, since we do not know whether the constructed metrics have at least one orthogonal bi-invariant complex structure, which is needed to apply Theorem \ref{thm:main}.


First we discuss decompositions of Lie algebras without inner product and compare it to the factors of metric Lie algebras. Similarly as for metric Lie algebras, we call a Lie algebra without inner product indecomposable if is has no non-trivial decomposition.
\begin{Def}
	A Lie algebra $\lie$ is called \textbf{indecomposable} if it cannot be written as a direct sum of ideals $\lie = \lie_1 \oplus \lie_2$ with $\lie_1 \neq 0 \neq \lie_2$. 
\end{Def}

\noindent It is immediate that an indecomposable Lie algebra is irreducible for every inner product and vice versa, if a Lie algebra is irreducible for every inner product it must be indecomposable. 

In contrast, given a metric Lie algebra which is irreducible, it does not hold that the Lie algebra without inner product must be indecomposable. In this section we will not only present concrete examples of such metric Lie algebras, but also show that on every Lie algebra without abelian factors there exists an inner product making it irreducible, showing this is a common phenomenon.

\begin{Thm}
	\label{thm:indecomposable}
Let $\lie$ be a real or complex Lie algebra without no non-zero abelian factors, then there exists an inner product $\langle \cdot, \cdot \rangle$ on $\lie$ such that the corresponding metric Lie algebra $(\lie,\langle \cdot, \cdot \rangle)$ is irreducible. 
\end{Thm}

The main ingredient is \cite[Theorem 3.3.]{fgh13-1} which describes the possibilities for decomposing a Lie algebra into indecomposable factors. We present it in a simplified form here, which is sufficient for our purposes.

\begin{Thm}
	\label{thm:decompositionwithoutmetric}
	Let $\lie$ be a Lie algebra with two decompositions $$\lie = \lie_1 \oplus \ldots \oplus \lie_k = \lieh_1 \oplus \ldots \oplus \lieh_l$$ into indecomposable components. The number of factors $k = l$ is the same and, up to renumbering the $\lieh_j$, the factors $\lie_j \approx \lieh_j$ are isomorphic with identical commutator subalgebras $[\lie_j,\lie_j] = [\lieh_j,\lieh_j]$. 
\end{Thm}

The main consequence of this theorem is not only that the indecomposable factors are unique up to isomorphism, but also that their commutator subalgebras form the same subalgebras of $\lie$. In particular, given the Lie algebra $\lie = \displaystyle \bigoplus_{j=1}^k \lie_j$ with the $\lie_j$ indecomposable, the commutator subalgebra $[\lie,\lie] $ has a unique decomposition $[\lie,\lie] = \displaystyle \bigoplus_{j=1}^k [\lie_j,\lie_j]$ into ideals, which do not depend on the choice of the indecomposable $\lie_j$. Note that the ideals $[\lie_j,\lie_j] \subset \lie$ could be decomposable, for example in the case where $\lie$ is $2$-step nilpotent and thus $[\lie,\lie]$ is abelian. If we restrict ourselves to Lie algebras where none of the $\lie_j$ is abelian, then every factor $\lie_j$ corresponds to a unique factor $[\lie_j,\lie_j] \neq 0$ of the decomposition of $[\lie,\lie]$. 

We start by an easy lemma about the existence of inner products on vector spaces.

\begin{Lem}
\label{lem:innerproduct}
Let $V = \displaystyle \bigoplus_{j\in I} V_j$ be a real or complex vector space given as a direct sum of subspaces $0 \neq V_j \subset V$. There exists an inner product $\langle \cdot,\cdot \rangle$ such that for every disjoint union $I = I_1 \cup I_2$ with $I_1 \neq I \neq I_2$, it holds that the subspaces $\displaystyle\bigoplus_{j \in I_1} V_j$ and $\displaystyle\bigoplus_{j\in I_2} V_j$ are not orthogonal.
\end{Lem}
\begin{proof}
Note that the space of inner products on a vector space of dimension $n$ is a manifold of either dimension $\frac{n(n+1)}{2}$ in the real case or dimension $n^2$ in the complex case. In particular, the space of inner products on $V = W_1 \oplus W_2$ with $W_1 \neq 0 \neq W_2$ making $W_1$ and $W_2$ orthogonal has dimension strictly smaller than the space all inner products on $V$. The lemma follows by applying this observation to all possibilities for $I = I_1 \cup I_2$ and the fact that a manifold cannot be covered by finitely many submanifolds of strictly lower dimension.
\end{proof}

\noindent It is an easy exercise to give an explicit form for an inner product as in Lemma \ref{lem:innerproduct}, but since we only need the existence, we stated it as above.

\begin{proof}[Proof of Theorem \ref{thm:indecomposable}]
	Take any decomposition $$\lie = \lie_1 \oplus \ldots \oplus \lie_k$$ into indecomposable factors. By Theorem \ref{thm:decompositionwithoutmetric} and the fact that $\lie$ has no non-zero abelian factors, we get a decomposition $[\lie,\lie] = [\lie_1,\lie_1] \oplus \ldots \oplus [\lie_k,\lie_k]$ with $[\lie_j,\lie_j] \neq 0$ for every $1 \leq j \leq k$. Now take any inner product $\langle \cdot, \cdot \rangle$ on $\lie$ such that the restriction on $[\lie,\lie]$ satisfies the conclusion of Lemma \ref{lem:innerproduct} for the decomposition $[\lie,\lie] = \displaystyle \bigoplus_{j=1}^k [\lie_j,\lie_j]$.  
	We claim that $\lie$ is irreducible with this inner product. 
	
	Indeed, assume that this is not the case and that $\lie = \lien_1 \oplus \lien_2$ has an orthogonal decomposition. By taking decompositions $\lien_1 = \lieh_{1} \oplus \ldots \oplus \lieh_{l}$ and $\lien_2 = \lieh_{l+1} \oplus \ldots \oplus \lieh_k$ into indecomposable factors, we get also a decomposition $$\lie = \lien_1 \oplus \lien_2 = \lieh_1 \oplus \ldots \oplus \lieh_k$$ of $\lie$ into indecomposable ideals. Since $\lien_1$ is orthogonal to $\lien_2$, also $[\lien_1,\lien_1]$ is orthogonal to $[\lien_2,\lien_2]$. Because $$[\lie,\lie] = [\lien_1,\lien_1] \oplus [\lien_2,\lien_2] = \left(\bigoplus_{j=1}^l [\lieh_j,\lieh_j] \right)\oplus \left( \bigoplus_{i=l+1}^k[\lieh_j,\lieh_j] \right)$$ and the assumption on the inner product, we get that either $[\lien_1,\lien_1] = 0$ or $[\lien_2,\lien_2] = 0$ and thus also either $\lien_1 = 0$ or $\lien_2 = 0$.
\end{proof}

The proof even shows that the inner products for which $\lie$ is not irreducible are rather scarce. We give a concrete example of an irreducible metric Lie algebra which does not come from an indecomposable Lie algebra.
\begin{Ex}
	\label{ex:irreduciblemetric}
	Take the Lie algebra $\lie = \lieh_3(\R) \oplus \lieh_3(\R)$ with basis $X_1, X_2, Y_1, Y_2, Z_1, Z_2$ and relations $[X_j,Y_j]=Z_j$ for $j \in \{1,2\}$. Now consider the inner product with orthonormal basis $X_1, \hspace{1mm} X_2, \hspace{1mm} Y_1, \hspace{1mm} Y_2, \hspace{1mm} Z_1$ and $Z_1 -Z_2$. By Theorem \ref{thm:decompositionwithoutmetric} any decomposition $\lie = \lie_1 \oplus \lie_2$ with $\lie_1 \neq 0 \neq \lie_2$ has the property that $[\lie_j,\lie_j]$ is spanned by either $Z_1$ or $Z_2$. Since $\langle Z_1, Z_2 \rangle = 1 \neq 0$, we conclude that $\lie$ is irreducible with this metric, although it is not indecomposable.
\end{Ex}

As a consequence, we show that on a Lie algebra given by $k$ indecomposable ideals, the number of irreducible factors can take any number between $1$ and $k$ by varying the inner product.
\begin{Cor}
	\label{cor:numberirr}
Let $\lie$ be a Lie algebra with a decomposition $$\lie = \bigoplus_{j=1}^k \lie_j$$ consisting of $k$ indecomposable ideals $\lie_j$, such that $[\lie_j,\lie_j] \neq 0$. For every $1 \leq l \leq k$ there exists an inner product $\langle \cdot, \cdot \rangle_l$ on $\lie$ such that it has exactly $l$ irreducible components. 	
\end{Cor}
\noindent Clearly, the theorem fails when admitting abelian factors.
\begin{proof}
Apply Theorem \ref{thm:indecomposable} on the Lie algebra $\displaystyle \bigoplus_{j=l}^{k} \lie_j$ to find an inner product which makes it irreducible. Take any inner product on the Lie algebras $\lie_1, \ldots, \lie_{l-1}$, then by extending it orthogonally to $\lie$ we get the result. 
\end{proof}

We have shown in Corollary \ref{cor:finitenumber} that the number of orthogonal bi-invariant complex structures is either $0$ or $2^k$ with $k$ the number of irreducible components. By Corollary \ref{cor:numberirr} the number of irreducible components varies with the inner product, leading to the following theorem.

\begin{Thm}
	\label{Q}
Let $\lie = \displaystyle \bigoplus_{j=1}^k \lie_j$ be a complex Lie algebra given as a direct sum of indecomposable ideals $\lie_j$ with $[\lie_j,\lie_j] \neq 0$. For every $1 \leq l \leq k$ there exists an inner product on the underlying Lie algebra of $\lie$ such that it has exactly $2^l$ orthogonal bi-invariant complex structures.
\end{Thm}
\begin{proof}
Take any inner product $\langle \cdot,\cdot\rangle_l$ on $\lie$ such that it has exactly $l$ irreducible factors, which exists by Corollary \ref{cor:numberirr}. Since the underlying Lie algebra of $\lie$ has at least one orthogonal bi-invariant complex structure, given by multiplication by $i$, it has exactly $2^l$ orthogonal bi-invariant complex structures by Theorem \ref{thm:main}.
\end{proof}

The theorem starts from a complex Lie algebra and its number of indecomposable components. Since Theorem \ref{thm:irreducible} only works for metric Lie algebras, it is unclear at the moment how this relates to the decomposition of the underlying real Lie algebra into indecomposable factors. Also, it is not known whether every real Lie algebra has an inner product such that it does not admit any orthogonal bi-invariant complex structure.

\begin{QN}
	\label{Q:varyingmetric}
	Let $\lie = \displaystyle \bigoplus_{j=1}^k \lie_j$ be a real Lie algebra given as a direct sum of indecomposable ideals $\lie_j$ with $[\lie_j,\lie_j] \neq 0$. What are the possible number of orthogonal bi-invariant complex structures on $\lie$ for different inner products on $\lie$? Does there always exist an inner product such that $\lie$ has no orthogonal bi-invariant complex structures?
\end{QN}

If the dimension is even, every inner product on a real abelian Lie algebras admits an orthogonal bi-invariant complex structure. The following low-dimensional example shows that for nilpotent Lie algebras this is not the case.
\begin{Ex}	Consider the underlying real Lie algebra $\lie$ of the complex Heisenberg algebra $\lieh_3(\C)$ with $X, Y$ and $Z$ the standard complex basis for $\lieh_3(\C)$ satisfying $[X,Y] = Z$. As a real Lie algebra with basis $E_1 = X, E_2 = iX, E_3 = Y, E_4 = iY, E_5 = Z$ and $E_6 = iZ$, the structure constants are given by 
	\begin{align*}
	[E_1,E_3] &= E_5 \\
	[E_2,E_4] &= -E_5 \\
	[E_2,E_3] &= E_6 \\
	[E_1,E_4] &= E_6.
	\end{align*}
	Of course there are inner products for which $\lie$ admits two orthogonal bi-invariant complex structures, starting from any Hermitian inner product on $\lieh_3(\C)$. Since $\lie$ is indecomposable, every inner product makes it irreducible, hence $\lie$ can have at most $2$ orthogonal bi-invariant complex structures for any inner product.
	
	On the other hand, if $\langle \cdot, \cdot \rangle$ is an inner product on $\lie$ such that $E_5$ and $E_6$ are not orthogonal, then $\lie$ does not admit an orthogonal bi-invariant complex structure. Indeed, if $J$ is a bi-invariant complex structure on $\lie$, it will commute with multiplication by $i$ on $[\lie,\lie]$ by definition of bi-invariant complex structure, exactly as in the proof of Proposition \ref{prop:commutegeneral}. This implies that there exists $a,b \in \R$ such $J(E_5) = aE_5 + b E_6$ and $J(E_6) = -b E_5 + a E_6$. Since $J^2 = -\I_\lie$, a computation shows that $a=0$ and $b \in \{ \pm 1\}$. Because $E_5$ and $J(E_5) = \pm E_6$ are not orthogonal, we conclude that such a $J$ does not exist.
\end{Ex} 

If the dimension of $[\lie,\lie]$ is greater than $2$, it is not clear how to generalize this example, leaving the last part of Question \ref{Q:varyingmetric} open.

\bibliography{ref}
\bibliographystyle{plain}

\end{document}